\pdfoutput=1
\documentclass[DIV=11]{scrartcl}
\usepackage{amsmath,amsthm,amssymb}
\usepackage{breqn,bm,ifpdf,authblk,extpfeil,scalerel,relsize,accents,fontawesome}
\usepackage[cal=boondoxo]{mathalfa}
\usepackage[abbrev,nobysame,lite]{amsrefs}
\usepackage[dvipsnames]{xcolor}
\usepackage[T1]{fontenc}
\usepackage{lmodern}
\usepackage[inline,shortlabels]{enumitem}
\usepackage{microtype}
\usepackage{hyperref}

\allowdisplaybreaks[1]

\makeatletter
\DeclareOldFontCommand{\rm}{\normalfont\rmfamily}{\mathrm}
\DeclareOldFontCommand{\sf}{\normalfont\sffamily}{\mathsf}
\DeclareOldFontCommand{\bf}{\normalfont\bfseries}{\mathbf}
\DeclareOldFontCommand{\it}{\normalfont\itshape}{\mathit}
\makeatother

\newtheorem{thm}{Theorem}[section]
\newtheorem{lemma}[thm]{Lemma}
\newtheorem{prop}[thm]{Proposition}
\newtheorem{cor}[thm]{Corollary}

\newtheorem{thmabc}{Theorem}
\newtheorem{corabc}[thmabc]{Corollary}
\newtheorem{propabc}[thmabc]{Proposition}

\theoremstyle{definition}

\newtheorem{ex}[thm]{Example}
\newtheorem{rem}[thm]{Remark}
\newtheorem{question}[thm]{Question}

\numberwithin{equation}{section}
\allowdisplaybreaks[1]

\renewcommand{\le}{\leqslant}
\renewcommand{\ge}{\geqslant}

\renewcommand{\geq}{\ge}

\renewcommand{\qedsymbol}{\ensuremath{\blacklozenge}}

\def\emph{}
\DeclareTextFontCommand{\bfemph}{\bfseries}
\DeclareTextFontCommand{\itemph}{\itshape}
\def\emph{\bfemph}

\makeatletter
\def\blankfootnote{\xdef\@thefnmark{}\@footnotetext}

\newcommand*{\textlabel}[2]{%
  \edef\@currentlabel{#1}
  \phantomsection
  #1\label{#2}
}
\makeatother

\newcommand{\Biv}{\ensuremath{\mathcal{C}}}
\newcommand{\CSP}{\ensuremath{\mathsf F}}

\DeclareMathOperator{\cc}{cc}
\DeclareMathOperator{\concnt}{k}

\newcommand{\xto}{\xrightarrow}
\newcommand{\std}{\ensuremath{\mathsf{e}}}
\DeclareMathOperator{\supp}{supp}
\DeclareMathOperator{\identity}{id}
\DeclareMathOperator{\pred}{p}

\DeclareMathOperator{\Irr}{Irr}
\DeclareMathOperator{\GL}{GL}
\DeclareMathOperator{\Coker}{Coker}
\DeclareMathOperator{\Uni}{U}
\DeclareMathOperator{\Mat}{M}
\DeclareMathOperator{\rank}{rk}
\DeclareMathOperator{\Adj}{Adj}
\newcommand{\card}[1]{\lvert#1\rvert}
\DeclarePairedDelimiter{\abs}{\lvert}{\rvert}

\newcommand{\divides}[2]{\ensuremath{ {#1} \mid {#2} }}

\DeclareMathOperator{\CG}{K}
\DeclareMathOperator{\DG}{\Delta}
\DeclareMathOperator{\Star}{Star}
\DeclareMathOperator{\compcnt}{c}
\DeclareMathOperator{\nbhcnt}{d}
\DeclareMathOperator{\Nbh}{N}
\newcommand{\ConnDoms}{\ensuremath{\mathfrak D}^{\mathrm c}}
\newcommand{\join}{\ensuremath{\vee}}
\newcommand{\Path}[1]{\operatorname{P}_{{#1}}}

\newcommand{\Point}{\ensuremath{\bullet}}

\newcommand{\QQ}{\mathbf{Q}}
\newcommand{\FF}{\mathbf{F}}
\newcommand{\GG}{\mathbf{G}}
\newcommand{\HH}{\mathbf{H}}
\newcommand{\ZZ}{\mathbf{Z}}

\newcommand{\fh}{\ensuremath{\mathfrak h}}
\newcommand{\fz}{\ensuremath{\mathfrak z}}

\newcommand{\Tau}{\ensuremath{\mathsf{T}}}

\newcommand{\fO}{\mathfrak{O}}
\newcommand{\fP}{\mathfrak{P}}
\newcommand{\sA}{\mathsf{A}}
  
\ifpdf
  \hypersetup{pdftitle={Enumerating conjugacy classes of graphical groups over finite fields},
    pdfauthor={Tobias Rossmann}
  }
\fi

\title{Enumerating conjugacy classes of graphical groups over finite fields}

\author{Tobias Rossmann}

\date{}

\begin{document}
\thispagestyle{empty}

\maketitle

\thispagestyle{empty}
\vspace*{-2em}
\begin{abstract}
  \small
  Each graph and choice of a commutative ring gives rise to an associated
  graphical group.
  In this article, we introduce and investigate graph polynomials that enumerate
  conjugacy classes of graphical groups over finite fields according to their
  sizes.
\end{abstract}

\blankfootnote{\noindent{\itshape 2020 Mathematics Subject Classification.}
  20D15, 
  20E45, 
  05A15, 
  05C25, 
  05C31 

}

\section{Introduction}
\label{s:intro}

\subsection{Graphical groups}
\label{ss:intro/graphical}

Throughout, graphs are finite, simple, and (unless otherwise indicated)
contain at least one vertex. 
When the reference to an ambient graph is clear, we use $\sim$ to indicate the
associated adjacency relation.
All rings are associative, commutative, and unital.

Let $\Gamma = (V,E)$ be a graph with $n$ vertices.
The \emph{graphical group} $\GG_\Gamma(R)$ associated with $\Gamma$ over a
ring $R$ was defined in \cite[\S 3.4]{cico}.
For a short equivalent description (see~\S\ref{ss:relating_graphical}), write
$V = \{v_1,\dotsc,v_n\}$ and let $J = \{ (j,k) : 1\le j < k \le n,\, v_j\sim
v_k\}$.
Then $\GG_\Gamma(R)$ is generated by symbols $x_1(r),\dotsc,x_n(r)$ and
$z_{jk}(r)$ for $(j,k)\in J$ and $r\in R$, subject to the following defining
relations for $i,i' \in [n] := \{1,\dotsc,n\}$, $(j,k), (j',k')\in J$, and
$r,r'\in R$:

\begin{enumerate}
\item
  $x_i(r)  x_i(r') = x_i(r+r')$ and $z_{jk}(r)z_{jk}(r') = z_{jk}(r+r')$.
  \hfill
  (``scalars'')
\item
  $[x_j(r),x_k(r')] = z_{jk}(rr')$.
  \hfill
  (``adjacent vertices and commutators'')

  (Recall that $(j,k) \in J$ so that $v_j \sim v_k$.)
\item
  $[x_i(r),x_{i'}(r')] = 1$ if $v_i\not\sim v_{i'}$.
  \hfill
  (``non-adjacent vertices and commutators'')
\item
  $[x_i(r), z_{jk}(r')] = [ z_{j'k'}(r),z_{jk}(r')]= 1$.
  \hfill
  (``centrality of commutators'')
\end{enumerate}

Note that every ring map $R\to R'$ induces an evident group homomorphism
$\GG_\Gamma(R) \to \GG_\Gamma(R')$.
We will see in \S\ref{ss:relating_graphical} that the resulting group functor
$\GG_\Gamma$ represents the \emph{graphical group scheme} associated with
$\Gamma$ as defined in \cite{cico}.
The isomorphism type of $\GG_\Gamma$ does not depend on the chosen ordering of
the vertices of $\Gamma$.

\begin{ex}
  \label{ex:graphical_relatives}
  Various instances and relatives of graphical groups appeared in the
  literature.
  \begin{enumerate}[(i)]
  \item
    \label{ex:graphical_relatives_artin}
    $\GG_\Gamma(\ZZ)$ is isomorphic to the maximal nilpotent quotient of class
    at most $2$ of the right-angled Artin group $\langle x_1,\dotsc,x_n \mid
    [x_i,x_j] = 1  \text{ whenever } v_i\not\sim v_j\rangle$;
    see \S\ref{ss:artin}.
\item
  \label{ex:graphical_relatives_complete}
  Let $\CG_n$ denote a complete graph on $n$ vertices.
  Then $\GG_{\CG_n}(\ZZ)$ is a free nilpotent group group of rank $n$ and
  class at most $2$.
  For each odd prime~$p$, the graphical group $\GG_{\CG_n}(\FF_p)$ is a free
  nilpotent group of rank $n$, exponent dividing~$p$, and class at most $2$.
  (Both statements follow from Proposition~\ref{prop:ZZ_and_ZZk_points}
  below.
  We note that the prime $2$ does not play an exceptional role in any of our main
  results.)
\item
  \label{ex:graphical_relatives_path}
  Let $\Path n$ be a path graph on $n$ vertices.
  Let $\Uni_d\le \GL_d$ be the group scheme of upper unitriangular $d\times d$
  matrices.
  Then for each ring $R$, the group $\GG_{\Path n}(R)$ is the maximal quotient
  of class at most $2$ of $\Uni_{n+1}(R)$; cf.~\cite[\S 9.4]{cico}.
\item
  \label{ex:graphical_relatives_discrete}
  Let $\DG_n$ denote an edgeless graph on $n$ vertices.
  Then for each ring $R$, we may identify $\GG_{\DG_n}(R)$ and the (abelian)
  additive group $R^n$.
\item
  \label{ex:graphical_relatives_baer}
  Let $p$ be an odd prime.
  Then $\GG_\Gamma(\FF_p)$ is isomorphic to the $p$-group attached to the
  complement of $\Gamma$ via Mekler's construction~\cite{Mek81};
  cf.~Proposition~\ref{prop:ZZ_and_ZZk_points}\ref{prop:ZZ_and_ZZk_points2}.
  Li and Qiao~\cite{LQ20} used what they dubbed the {Baer-Lov\'asz-Tutte
    procedure} to attach a finite $p$-group to $\Gamma$.
  Their group is also isomorphic to $\GG_\Gamma(\FF_p)$; see \S\ref{ss:centralisers}.
\end{enumerate}
\end{ex}

\subsection{Known results: class numbers of graphical groups}
\label{ss:intro/known}

Let $\cc_e(G)$ denote the number of conjugacy classes of size $e$ of a finite
group $G$, and let $\concnt(G) = \sum_{e=1}^\infty \cc_e(G)$ be the
\emph{class number} of $G$.
It is well known that $\concnt(\GL_d(\FF_q))$ is a polynomial in $q$ for fixed
$d$; see \cite[Ch.~1, Exercise~190]{Sta12}.
This article is devoted to the class numbers $\concnt(\GG_\Gamma(\FF_q))$.
We first recall known results.

\begin{thm}[{\cite[Cor.~1.3]{cico}}]
  \label{thm:cico/poly}
  Given a graph $\Gamma$, there exists $f_\Gamma(X)\in \QQ[X]$ such that
  $\concnt(\GG_\Gamma(\FF_q)) = f_\Gamma(q)$ for each prime power $q$.
\end{thm}

We call $f_\Gamma(X)$ the \emph{class-counting polynomial} of~$\Gamma$.
In \cite{cico}, Theorem~\ref{thm:cico/poly} is derived from
a more general uniformity result~\cite[Cor.~B]{cico} for class-counting zeta
functions associated with graphical group schemes; see \S\ref{ss:cc_zetas}.
Formulae for $f_\Gamma(X)$ when $\Gamma$ has at most $5$ vertices can be
deduced from the tables in \cite[\S 9]{cico}.
Moreover, several families of class-counting polynomials have been previously
computed in the literature.

\begin{ex}
  \label{ex:univ_CG}
  O'Brien and Voll~\cite[Thm~2.6]{O'BV15} gave a formula for the number of
  conjugacy classes of given size of $p$-groups derived from free nilpotent
  Lie algebras via the Lazard correspondence.
  Using the interpretation of $\GG_{\CG_n}(\FF_p)$ in
  Example~\ref{ex:graphical_relatives}\ref{ex:graphical_relatives_complete},
  their formula or, alternatively, work of Ito and Mann~\cite[\S 1]{IM06}
  yields $f_{\CG_n}(X) = X^{\binom {n-1} 2}(X^n + X^{n-1}-1)$.
\end{ex}

\begin{ex}
  \label{ex:paths}
  In light of
  Example~\ref{ex:graphical_relatives}\ref{ex:graphical_relatives_path},
  Marjoram's enumeration \cite[Thm~7]{Mar99} of the irreducible characters of
  given degree of the maximal class-$2$ quotients of $\Uni_d(\FF_q)$ yields
  \begin{equation}
    \label{eq:f_path}
    \displaystyle
    f_{\Path n}(X) =
    \mathlarger{\sum}_{a=0}^{\left\lfloor \frac n 2\right\rfloor}
    \left(\binom{n-a} a X^{n-a-1}(X-1)^a
    + \binom{n-a-1}aX^{n-a-1}(X-1)^{a+1}\right). 
  \end{equation}
\end{ex}

For any graph $\Gamma$, the size of each conjugacy class of
$\GG_\Gamma(\FF_q)$ is of the form~$q^i$; see
Proposition~\ref{prop:centralisers}\ref{prop:centralisers1}.
As indicated in \cite[\S 8.5]{cico}, the methods underpinning
Theorem~\ref{thm:cico/poly} can be used to strengthen said theorem: each
$\cc_{q^i}(\GG_\Gamma(\FF_q))$ is a polynomial in $q$ with rational
coefficients.
While constructive, the proof of Theorem~\ref{thm:cico/poly} in \cite{cico}
relies on an elaborate recursion.
In particular, no explicit general formulae for the numbers
$\cc_{q^i}(\GG_\Gamma(\FF_q))$ or the polynomial $f_\Gamma(X)$ have been
previously recorded.

Recall that the \emph{join}\label{d:join} $\Gamma_1\join \Gamma_2$ of graphs
$\Gamma_1$ and~$\Gamma_2$ is obtained from their disjoint union
$\Gamma_1\oplus \Gamma_2$ by adding edges connecting each vertex of $\Gamma_1$
to each vertex of $\Gamma_2$.
Further recall that a \emph{cograph} is any graph that can be obtained from
two cographs on fewer vertices by taking disjoint unions or joins, starting
with an isolated vertex.

\begin{thm}[{Cf.\ \cite[Theorem E]{cico}}]
  \label{thm:cico/cograph_nonneg}
  Let $\Gamma$ be a cograph.
  Then the coefficients of $f_\Gamma(X)$ as a polynomial in $X-1$ are
  non-negative integers.
\end{thm}

Theorems~\ref{thm:cico/poly} and \ref{thm:cico/cograph_nonneg} are special
cases of more general results pertaining to class numbers of graphical groups
$\GG_\Gamma(\fO/\fP^i)$, where $\fO$ is a compact discrete valuation ring with
maximal ideal $\fP$.
We will briefly discuss this topic in \S\ref{s:app}.

\subsection[The graph polynomials]{The graph polynomials $\Biv_\Gamma(X,Y)$ and $\CSP_\Gamma(X,Y)$}

As before, let $\Gamma = (V,E)$ be a graph.
Prior to stating our results, we first define what, to the author's knowledge,
appears to be a new graph polynomial.
For a (not necessarily proper) subset $U\subset V$, let $\Gamma[U]$ be the
induced subgraph of $\Gamma$ with vertex set $U$.
Let $\compcnt_\Gamma(U)$ denote the number of connected components of
$\Gamma[U]$.
(We allow $U = \emptyset$ in which case $\compcnt_\Gamma(U) = 0$.)
The \emph{closed neighbourhood} $\Nbh_\Gamma[v] \subset V$ of 
$v\in V$ consists of~$v$ and all vertices adjacent to it.
For $U\subset V$, write $\Nbh_\Gamma[U] = \bigcup_{u\in U}\Nbh_\Gamma[u]$.
Define
\begin{equation}
  \label{eq:Biv}
  \Biv_\Gamma(X,Y) =
  \sum_{U\subset V} (X-1)^{\card{U}}\, Y^{\card{\Nbh_\Gamma[U]} - \compcnt_\Gamma(U)}
  \in \ZZ[X,Y].
\end{equation}

\begin{rem}
  While $\Biv_\Gamma(X,Y)$ resembles the Tutte polynomial of a matroid on the
  ground set $V$, it is unclear to the author whether this is more than a
  formal similarity.
  Similarly, $\Biv_\Gamma(X,Y)$ is reminiscent of the subgraph
  polynomial~\cite[\S 3]{Ren02} of $\Gamma$.
\end{rem}

Let $\Gamma$ have $n$ vertices, $m$ edges, and $c$ connected components.
Recall that the \emph{(matroid) rank} of~$\Gamma$ is $\rank(\Gamma) = n-c$.
Define the \emph{class-size polynomial} of $\Gamma$ to be
\begin{equation}
  \label{eq:csp}
  \CSP_\Gamma(X,Y) = X^m \Biv_\Gamma(X,X^{-1}Y)
  = \sum_{U\subset V} (X-1)^{\card U} X^{m + \compcnt_\Gamma(U)
    -\card{\Nbh_\Gamma[U]}} Y^{\card{\Nbh_\Gamma[U]}-\compcnt_\Gamma(U)}.
\end{equation}
The degree of $\Biv_\Gamma(X,Y)$ as a polynomial in $Y$ is $\rank(\Gamma)$;
see Proposition~\ref{prop:deg_Y}.
As $\rank(\Gamma) \le m$, we conclude that $\CSP_\Gamma(X,Y) \in \ZZ[X,Y]$.
While $\Biv_\Gamma(X,Y)$ and $\CSP_\Gamma(X,Y)$ determine each other,
$\Biv_\Gamma(X,Y)$ is often more convenient to work with and
$\CSP_\Gamma(X,Y)$ turns out to be more directly related to the enumeration of
conjugacy classes; see Theorem~\ref{thm:formula}.

\begin{ex}
  \label{ex:CG_DG}
  \quad
  \begin{enumerate}[(i)]
  \item
    \label{ex:CG_DG_1}
    $\Biv_{\CG_n}(X,Y) = (X^n -1) Y^{n-1} + 1$
    and 
    $\CSP_{\CG_n}(X,Y) = \Bigl(X^{\binom n 2+1}-X^{\binom{n-1}2}\Bigr)Y^{n-1}
    + X^{\binom n2}$.
  \item
    \label{ex:CG_DG_2}
    $\Biv_{\DG_n}(X,Y) = X^n = \CSP_{\DG_n}(X,Y)$.
  \end{enumerate}
\end{ex}

\subsection{Main results}

Let $\Gamma$ be a graph.
The main result of this article justifies the term ``class-size polynomial''.

\begin{thmabc}
  \label{thm:formula}
  $\CSP_\Gamma(q,Y) = \sum\limits_{i=0}^\infty \cc_{q^i}(\GG_\Gamma(\FF_q)) Y^i$
  for each prime power $q$.
\end{thmabc}

Note that $\cc_{q^i}(\GG_\Gamma(\FF_q)) = 0$ for all sufficiently large $i$
so Theorem~\ref{thm:formula} asserts an equality of polynomials in $Y$.

\begin{ex}
    The formula in \cite[Thm~2.6]{O'BV15} referred to in
    Example~\ref{ex:univ_CG} shows that if $q$ is an odd prime power, then
    $\GG_{\CG_n}(\FF_q)$ has a centre of order $q^{\binom{n} 2}$ and precisely
    $(q^n-1)q^{\binom{n-1} 2}$ non-trivial conjugacy classes, all of size
    $q^{n-1}$.
    These numbers agree with Example~\ref{ex:CG_DG}\ref{ex:CG_DG_1}.
    Clearly,
    Example~\ref{ex:CG_DG}\ref{ex:CG_DG_2} agrees
    with the fact that $\GG_{\DG_n}(\FF_q) \approx \FF_q^n$ is abelian.
\end{ex}

Theorem~\ref{thm:formula} provides us with the following explicit formula for the
class-counting polynomial $f_\Gamma(X)$ defined in Theorem~\ref{thm:cico/poly}.

\begin{corabc}
  \label{cor:univ_formula}
  $\displaystyle f_\Gamma(X) 
  = \CSP_\Gamma(X,1)
  = \sum_{U\subset V} (X-1)^{\card U}\, X^{m+\compcnt_\Gamma(U)-\card{\Nbh_\Gamma[U]}}$.
  \qed
\end{corabc}

Note that Corollary~\ref{cor:univ_formula} shows that $f_\Gamma(X)$ has
integer coefficients.
In the spirit of work surrounding Higman's conjecture (see
\S\ref{ss:intro/higman}) and Theorem~\ref{thm:cico/cograph_nonneg},
Theorem~\ref{thm:formula} implies the following refinement of the preceding
observation.

\begin{corabc}
  \label{cor:nonneg}
  For each $e\ge 1$, the number of conjugacy classes of $\GG_\Gamma(\FF_q)$ of
  size $e$ is given by a polynomial in $q-1$ with non-negative integer
  coefficients.
\end{corabc}
\begin{proof}
  Use the binomial theorem to expand powers of $X = (X-1) + 1$ in \eqref{eq:csp}.
\end{proof}

It is natural to ask whether the coefficients referred to in
Corollary~\ref{cor:nonneg} enumerate meaningful combinatorial objects.
Corollary~\ref{cor:lc} will provide a partial answer to this.

We shall not endeavour to improve substantially upon the exponential-time
algorithm for computing $\Biv_\Gamma(X,Y)$ suggested by
equation~\eqref{eq:Biv}.
Indeed, we will obtain the following.

\begin{propabc}
  \label{prop:hard}
  Computing $\Biv_\Gamma(X,Y)$, and hence also $\CSP_\Gamma(X,Y)$, is NP-hard.
\end{propabc}

More precisely, we will see that knowledge of $\Biv_\Gamma(X,Y)$ allows us to
read off the cardinalities of connected dominating sets of $\Gamma$.
The problem of deciding whether a graph admits a connected dominating set of
cardinality at most a given number is known to be NP-complete;
see Theorem~\ref{thm:conn_dom_NP}.

By Theorem~\ref{thm:formula} and Proposition~\ref{prop:hard}, symbolically
enumerating the conjugacy classes of given size of $\GG_\Gamma(\FF_q)$ (as a
polynomial in~$q$) is NP-hard.
The problem of measuring the difficulty of symbolically enumerating
\itemph{all} conjugacy classes of $\GG_\Gamma(\FF_q)$ remains open.
\begin{question}
  \label{qu:computing_f}
  Is computing $f_\Gamma(X)$ NP-hard?
\end{question}

\subsection{Related work: around Higman's conjecture}
\label{ss:intro/higman}

Recall that $\Uni_d\le \GL_d$ denotes the group scheme of upper unitriangular
$d\times d$ matrices.
A famous conjecture due to G.~Higman~\cite{Hig60a} predicts that
$\concnt(\Uni_d(\FF_q))$ is given by a polynomial in $q$ for fixed $d$.
This has been confirmed for $d\le 13$ by Vera-L\'opez and Arregi~\cite{VLA03}
and for $d\le 16$ by Pak and Soffer~\cite{PS15}.
The former authors also showed that the sizes of conjugacy classes of
$\Uni_d(\FF_q)$ are of the form $q^i$ (see \cite[\S 3]{VLA92}) and that
$\cc_{q^i}(\Uni_d(\FF_q))$ is a polynomial in $q-1$ with non-negative integer
coefficients for $i \le d-3$ (see \cite{VLA01}).
Many authors studied variants of Higman's conjecture for unipotent groups
derived from various types of algebraic groups; see e.g.~\cite{GR09}.

While logically independent of the work described here, Higman's conjecture
(and the body of research surrounding it) certainly provided motivation for
topics considered and results obtained in this article
(e.g.\ Corollary~\ref{cor:nonneg}).

\subsection{Open problems: enumerating characters of graphical groups}
\label{ss:Irr}

Let $\Irr(G)$ denote the set of (ordinary) irreducible characters of a finite
group $G$.
It is well known that $\concnt(G) = \card{\Irr(G)}$ (see e.g.\ \cite[V, \S
5]{Hup67}), and the enumeration 
of irreducible characters of a group (according to their degrees) has often
been studied as a ``dual'' of the enumeration of conjugacy classes (according
to their sizes); see e.g.\ \cite{LS05,O'BV15,ask2}.
For odd $q$, \cite[Thm~B]{O'BV15} implies that the degree $\chi(1)$ of each
irreducible character $\chi$ of a graphical group $\GG_\Gamma(\FF_q)$ is of
the form $q^i$.

\begin{question}
  \label{qu:char_count}
  Let $\Gamma$ be a graph and let $i\ge 0$ be an integer.
  How does
  \[
    \mathrm{ch}(\Gamma,i;q) := \#\left\{ \chi\in\Irr(\GG_\Gamma(\FF_q)) : \chi(1) = q^i\right\} 
  \]
  depend on the prime power $q$?
\end{question}

It is known that $\mathrm{ch}(\Gamma,i;q)$ is a polynomial in $q$
for $\Gamma = \DG_n$ (trivially), $\Gamma = \Path n$ (by \cite[Thm~7]{Mar99}),
and $\Gamma = \CG_n$ (for odd $q$; by \cite[Prop.\ 2.4]{O'BV15}.

Let $v_1,\dotsc,v_n$ be the distinct vertices of a graph $\Gamma$.
Let $Y$ consist of algebraically independent variables $Y_{ij}$ (over $\ZZ$)
indexed by pairs $(i,j)$ with $1\le i <
j \le n$ and $v_i \sim v_j$.
Let $B_\Gamma(Y)$ be the antisymmetric $n\times n$ matrix
whose $(i,j)$ entry for $i < j$
is equal to $Y_{ij}$ if $v_i\sim v_j$ and zero otherwise.
(That is, $B_\Gamma(Y)$ is a generic antisymmetric matrix with support
constraints defined by $\Gamma$ as in \cite{cico}.)
Let $m$ be the number of edges of $\Gamma$.
Using an arbitrary ordering, relabel our variables as $Y = (Y_1,\dotsc,Y_m)$.
Then \cite[Thm~B]{O'BV15} shows that for odd $q$,
up to a factor given by an explicit power of $q$ (depending on $\Gamma$ and $i$),
$\mathrm{ch}(\Gamma,i;q)$ coincides with
$\#\!\left\{ y\in \FF_q^m : \rank_{\FF_q}(B_\Gamma(y)) = 2i\right\}$.
In our proof of Theorem~\ref{thm:formula} (see \S\ref{s:proof_formula}), the
number of conjugacy classes of $\GG_\Gamma(\FF_q)$ of given size is similarly
expressed in terms of the number of specialisations of given rank of a matrix
of linear forms.
In that setting, the latter enumeration can be carried out
explicitly using algebraic and graph-theoretic arguments.

It is unclear to the author whether such a line of attack could be used to
answer Question~\ref{qu:char_count}.
Work of Belkale and Brosnan~\cite[Thm~0.5]{BB03} on rank counts for
generic \itemph{symmetric} (rather than antisymmetric) matrices with support
constraints leads the author to suspect that the functions of $q$ considered
in Question~\ref{qu:char_count} might be rather wild as $\Gamma$ and $i$
vary.

\subsection{Overview}

In \S\ref{s:graphical}, we relate the definition of graphical groups from
\S\ref{ss:intro/graphical} to that from \cite{cico}.
Introduced in \cite{cico}, adjacency modules are modules over polynomial
rings whose specialisations are closely related
to conjugacy classes of graphical groups.
In \S\ref{s:Adj}, we determine the dimensions of such specialisations over
fields.
By combining this with work of O'Brien and Voll~\cite{O'BV15}, in
\S\ref{s:proof_formula}, we prove Theorem~\ref{thm:formula}.
In~\S\ref{s:operations}, we show that the polynomials $\Biv_{\Gamma}(X,Y)$ are
well-behaved with respect to joins of graphs.
In~\S\ref{s:const_lt}, we consider the constant term and leading coefficient
of $\Biv_\Gamma(X,Y)$ in $Y$ and we prove Proposition~\ref{prop:hard}.
Next, \S\ref{s:ccp} is devoted to the degree of $f_\Gamma(X)$.
Finally, in \S\ref{s:app}, we relate our findings to the study of zeta
functions enumerating conjugacy classes.

\subsection{\textit{Notation}}
\label{ss:notation}

The symbol ``$\subset$'' indicates not necessarily proper inclusion.
Group commutators are written $[x,y] = x^{-1}y^{-1}xy$.
For a ring $R$ and set $A$, $RA$ denotes the free $R$-module with basis
$(\std_a)_{a\in A}$.
For $x \in RA$, we write $x = \sum_{a\in A} x_a \std_a$.
We view $d\times e$ matrices over~$R$ as maps $R^d\to R^e$ acting by right
multiplication.
We let $\Point$ denote a graph with one vertex.

\section{Graphical groups and group schemes}
\label{s:graphical}

Throughout this section, let $\Gamma = (V,E)$ be a graph with $n$ vertices and
$m$ edges. 
We write $V = \{v_1,\dotsc,v_n\}$ and $J = \{ (j,k) : 1\le j < k \le n,
v_j\sim v_k\}$.
For a ring $R$, let $R\, \Gamma$ denote the free $R$-module of rank $m+n$ with
basis consisting of $\std_1,\dotsc,\std_n$ and all $\std_{jk}$ for $(j,k)\in
J$.
The chosen ordering of $V$ allows us to identify $R\,\Gamma = R V \oplus R E$
(see \S\ref{ss:notation}).

\subsection{Graphical groups over quotients of the integers}
\label{ss:artin}

Recall that the right-angled Artin group associated with the complement of
$\Gamma$ is
\[
  \sA_\Gamma := \left\langle x_1,\dotsc,x_n \mid [x_i,x_j] = 1 \text{ whenever }
    v_i\not\sim v_j \right\rangle.
\]
Let $\gamma_1(H) \geq \gamma_2(H) \ge \dotsb$ denote the lower central
series of a group $H$.
Recall the definition of $\GG_\Gamma(R)$ from \S\ref{ss:intro/graphical}.

\begin{prop}
  \label{prop:ZZ_and_ZZk_points}
  \quad
  \begin{enumerate}[(i)]
  \item
    \label{prop:ZZ_and_ZZk_points1}
    \textup{(Cf.~\cite[Rem.~3.8]{cico}.)} 
    $\GG_\Gamma(\ZZ) \approx \sA_\Gamma/\gamma_3(\sA_\Gamma)$.
  \item
    \label{prop:ZZ_and_ZZk_points2}
    $\GG_\Gamma(\ZZ/N\ZZ) \approx \sA_\Gamma/\gamma_3(\sA_\Gamma)\sA_\Gamma^N$
    if $N\ge 1$ is an odd integer.
  \end{enumerate}
\end{prop}

\begin{proof}
  Part~\ref{prop:ZZ_and_ZZk_points1} follows since $x_i(r) = x_i(1)^r$ and
  $z_{jk}(r) = z_{jk}(1)^r$ in $\GG_\Gamma(\ZZ)$ for $r\in\ZZ$, $i\in [n]$,
  and $(j,k)\in J$.
  Let $G = \langle X \rangle$ be a nilpotent group with $\gamma_3(G) = 1$.
  As is well known (and easy to see), $[ab,c] = [a,c][b,c]$ and
  $(ab)^N = a^Nb^N [b,a]^{\binom N 2}$ for $a,b,c\in G$;
  cf.\ \cite[III, Hilfssatz~1.2c) and Hilfssatz~1.3b)]{Hup67}.
  Let $N$ be odd so that $\divides N {\binom N 2}$,
  Then, if $x^N = 1$ for all $x\in X$, we find that $a^N = 1$ for all $a\in
  G$.
  Taking $X = \{x_1(1),\dotsc,x_n(1)\}$ and $G = \GG_\Gamma(\ZZ/N \ZZ)$,
  we obtain
  $\GG_\Gamma(\ZZ/N\ZZ)  \approx \GG_\Gamma(\ZZ)/\left\langle
    x_1(1)^N,\dotsc,x_n(1)^N\right\rangle
  \approx
  \GG_\Gamma(\ZZ)/\GG_\Gamma(\ZZ)^N \approx
  \sA_\Gamma/\gamma_3(\sA_\Gamma)\sA_\Gamma^N$, which proves
  \ref{prop:ZZ_and_ZZk_points2}.
\end{proof}

\subsection[Graphical group schemes]{Graphical group schemes following \cite{cico}}
\label{ss:cico/graphical}

We summarise the construction of the graphical group scheme $\HH_\Gamma$ from
\cite[\S 3.4]{cico} (denoted by $\GG_\Gamma$ in \cite{cico}).
For a ring $R$, the underlying set of the group $\HH_\Gamma(R)$ is $R\,
\Gamma$.
The group operation $*$ is characterised as follows:
\begin{enumerate}[(\textsf{G}1)]
\item
  $0\in R\, \Gamma$ is the identity element of $\HH_\Gamma(R)$.
\item
  For all $r_1,\dotsc,r_n\in R$, we have $r_1 \std_1 * \dotsb * r_n\std_n =
  r_1\std_1 + \dotsb + r_n \std_n$.
\item
  For $1 \le i \le j \le n$ and $r,s\in R$, we have
  \[
    s \std_j * r \std_i = \begin{cases}
      r \std_i + s \std_j - rs \std_{ij}, & \text{if } v_i\sim v_j,\\
      r\std_i + s\std_j, & \text{otherwise.}
    \end{cases}
  \]
\item
  For all $x\in R\,\Gamma$ and $z\in RE\subset R\,\Gamma$, we have $x*z = z*x = x+z$.
\end{enumerate}

Given a ring map $R \to R'$, the induced map $R\, \Gamma \to R'\, \Gamma$ is a
group homomorphism $\HH_\Gamma(R) \to \HH_\Gamma(R')$.
The resulting group functor $\HH_\Gamma$ represents the \emph{graphical group
scheme} constructed in \cite[\S 3.4]{cico}.

\subsection{Relating the two constructions of graphical group schemes}
\label{ss:relating_graphical}

The group functors $\GG_\Gamma$ (see \S\ref{ss:intro/graphical}) and
$\HH_\Gamma$ (see \S\ref{ss:cico/graphical}) are naturally isomorphic:

\begin{prop}
  \label{prop:two_graphical_group_schemes}
  For each ring $R$, the map $\theta_R\colon \HH_\Gamma(R) \to \GG_\Gamma(R)$
  given by
  \begin{align*}
    \sum_{i=1}^n r_i\std_i + \sum_{(j,k)\in J} r_{jk} \std_{jk}
    &\mapsto
    x_1(r_1)  \dotsb x_n(r_n) \prod_{(j,k)\in J} z_{jk}(r_{jk})
    &
    (r_i,r_{jk}\in R)
  \end{align*}
  is a group isomorphism.
  These maps combine to form a natural isomorphism of group functors
  $\HH_\Gamma\xto\approx \GG_\Gamma$.
\end{prop}
\begin{proof}
  By a simple calculation in $\HH_\Gamma(R)$, we find that for $1\le i < j \le
  n$ and $r_i,r_j\in R$,
  \[
    [r_i \std_i ,r_j \std_j] = \begin{cases}r_i r_j \std_{ij}, & \text{if } v_i\sim v_j,\\0,
      & \text{otherwise.}\end{cases}
  \]
  We thus obtain a group homomorphism $\pi_R\colon \GG_\Gamma(R) \to
  \HH_\Gamma(R)$ sending each $x_i(r)$ to $r\std_i$ and each $z_{jk}(r)$ to
  $r \std_{jk}$.
  By construction, $\pi_R \theta_R = \identity_{\HH_\Gamma(R)}$ and
  $\theta_R \pi_R = \identity_{\HH_\Gamma(R)}$.
\end{proof}

We are therefore justified in referring to both $\GG_\Gamma$ and $\HH_\Gamma$
as ``the'' graphical group scheme associated with $\Gamma$.
As a consequence of Proposition~\ref{prop:two_graphical_group_schemes}, each
$g\in \GG_\Gamma(R)$ admits a unique representation
\begin{align*}
  g & = x_1(r_1)\dotsb x_n(r_n) \prod_{(j,k)\in J}z_{jk}(r_{jk}).
  &
     (r_i,r_{jk}\in R)
\end{align*}
In particular, $\GG_\Gamma(R)$ has order $\card{R}^{m+n}$.

\subsection{Centralisers in graphical groups and graphical Lie algebras}
\label{ss:centralisers}

The \emph{graphical Lie algebra} $\fh_\Gamma(R)$ associated with $\Gamma$ over
a ring $R$ is defined by endowing the module $R\, \Gamma$ with the Lie bracket
$(\,\cdot,\cdot\,)$ characterised by the following properties:
\begin{itemize}[$\triangleright$]
\item
  For $1\le j<k\le n$, we have $(\std_{j},\std_{k})= \std_{jk}$ if $(j,k)\in
  J$ and $(\std_j,\std_k)= 0$ otherwise.
\item
  For $1\le i\le n$ and $(j,k),(j',k')\in J$, we have $(\std_i,\std_{jk}) =
  (\std_{jk},\std_{j'k'}) = 0$.
\end{itemize}

We may identify $\fh_\Gamma(R) = \fh_\Gamma(\ZZ) \otimes R$ as Lie
$R$-algebras and $\HH_\Gamma(R) = \fh_\Gamma(R)$ as sets.

Then $\HH_\Gamma$ is the group scheme associated with the Lie algebra
$\fh_\Gamma(\ZZ)$ via the construction from \cite[\S 2.4.1]{SV14}; cf.\
\cite[\S 2.4]{cico}.
In particular, if $2 \in R^\times$, then $\HH_\Gamma(R)$ (and hence
$\GG_\Gamma(R)$) is isomorphic to the group $\exp(\fh_\Gamma(R))$ associated
with $\fh_\Gamma(R)$ via the Lazard correspondence.
It follows that for an odd prime $p$, $\HH_\Gamma(\FF_p)$ is isomorphic to the
finite $p$-group attached to $\Gamma$ by Li and Qiao~\cite{LQ20}.
He and Qiao~\cite[Thm~1.1]{HQ21} showed that for graphs $\Gamma$ and $\Gamma'$
and an odd prime $p$, $\HH_\Gamma(\FF_p)$ and $\HH_{\Gamma'}(\FF_p)$ are
isomorphic if and only if $\Gamma$ and $\Gamma'$ are.

\begin{prop}
  \label{prop:centralisers}
  \quad
  \begin{enumerate}[(i)]
  \item
    \label{prop:centralisers1}
    The group centraliser of $h \in R\, \Gamma$ in $\HH_\Gamma(R)$ and the Lie
    centraliser of $h$ in $\fh_\Gamma(R)$ coincide as sets.
    Hence, the size of each conjugacy class of $\HH_\Gamma(\FF_q)$ is a power of $q$.
  \item
    \label{prop:centralisers2}
    The centres of $\HH_\Gamma(R)$ and $\fh_\Gamma(R)$ coincide as sets.
    The centre of $\fh_\Gamma(R)$ is the submodule of $R\, \Gamma$
    generated by all $\std_{jk}$ for  $(j,k)\in J$ and all $\std_i$ for
    {isolated} vertices~$v_i$.
  \item
    \label{prop:centralisers3}
    $[\HH_\Gamma(R),\HH_\Gamma(R)] = (\fh_\Gamma(R),\fh_\Gamma(R)) = RE$
    and $\HH_\Gamma(R)/[\HH_\Gamma(R),\HH_\Gamma(R)] \approx R V$.
  \end{enumerate}
\end{prop}
\begin{proof}
  The elements $r \std_i$ for $r\in R$ and $i=1,\dotsc,n$ generate
  $\HH_\Gamma(R)$ as a group and $\fh_\Gamma(R)$ as a Lie $R$-algebra.
  As $\HH_\Gamma(R)$ and $\fh_\Gamma(R)$ both have class at most $2$,
  using the calculation from the proof of
  Proposition~\ref{prop:two_graphical_group_schemes},
  we find that for $h_1,h_2\in R\,\Gamma$, the Lie bracket
  $(h_1,h_2)$ coincides with the group commutator $[h_1,h_2]$.
  All claims follow easily from this.
\end{proof}

\section{Adjacency modules}
\label{s:Adj}

Let $\Gamma = (V,E)$ be a graph.
Let $X_V = (X_v)_{v\in V}$ consist of algebraically independent variables over
$\ZZ$.
The \emph{adjacency module} of $\Gamma$ is the $\ZZ[X_V]$-module
\[
  \Adj(\Gamma) := \frac{\ZZ[X_V] V}{\langle X_v \std_w - X_w \std_v : v,w\in V
  \text{ with } v\sim w\rangle}.
\]
These modules were introduced in \cite[\S 3.3]{cico}.
Their study turns out to be closely related to the enumeration of conjugacy
classes of graphical groups; see \cite[\S\S 3.4,6,7]{cico}.
We note that what we call adjacency modules here are dubbed \itemph{negative}
adjacency modules in \cite{cico}.

For a ring $R$ and $x\in R V$, we obtain an $R$-module by specialising
$\Adj(\Gamma)$ in the form
\[
  \Adj(\Gamma)_x := \Adj(\Gamma) \otimes_{\ZZ[X_V]} R_x \approx \frac{ R V } {\langle x_v \std_w - x_w \std_v : v,w\in V
  \text{ with } v\sim w\rangle},
\]
where $R_x$ denotes $R$ regarded as a $\ZZ[X_V]$-algebra via
$X_v r = x_v r$ for $v\in V$ and $r\in R$.

\begin{lemma}
  \label{lem:Adj_disjoint_union}
  Let $\Gamma_i = (V_i,E_i)$ ($i=1,2$) be graphs on disjoint vertex sets.
  Let $\Gamma = \Gamma_1\oplus \Gamma_2 = (V,E)$
  be their disjoint union.
  Let $R$ be a ring and let $x\in R V$.
  Let $x_i\in R V_i$ denote the image of $x$ under the natural projection $R V
  = R V_1 \oplus  R V_2 \to R V_i$.
  Then $\Adj(\Gamma)_x \approx \Adj(\Gamma_1)_{x_1} \oplus
  \Adj(\Gamma_2)_{x_2}$ as $R$-modules.
  \qed
\end{lemma}

Let $K$ be a field.
For $x \in K V$, let $\supp(x) = \{v\in V : x_v\not= 0\}$.
Recall the definitions of $\Nbh_\Gamma[U]$ and $\compcnt_\Gamma(U)$ from
\S\ref{s:intro}.
The following is a key ingredient of our proof of Theorem~\ref{thm:formula}.

\begin{lemma}
  \label{lem:dim_Adjx}
  Let $x\in K V$
  and $U = \supp(x)$.
  Then
  \[
    \dim(\Adj(\Gamma)_x) = {\compcnt_\Gamma(U) + \card V -
      \card{\Nbh_\Gamma[U]}}.
  \]
\end{lemma}
\begin{proof}
  Let $H := \langle x_v\std_w -x_w\std_v : v\sim w\rangle \le K V$
  so that $\Adj(\Gamma)_x\approx K V/H$.
  \begin{enumerate}[(a)]
  \item
    \label{lem:dim_Adjx_a}
    Suppose that $\Gamma$ is connected and $U = V$.
    We need to show that $\dim(\Adj(\Gamma)_x) = 1$.
    To see that, first note that $H \subset x^\perp$, where the orthogonal
    complement is taken with respect to the bilinear form $y \cdot z =
    \sum_{v\in V} y_vz_v$.
    Hence, $\Adj(\Gamma)_x \not= 0$.
    Choose a spanning tree $\Tau$ of $\Gamma$ and a root $r\in V$.
    For $v\in V\setminus\{r\}$, let $\pred(v)$ be the predecessor of $v$ on the
    unique path from $r$ to $v$ in $\Tau$.
    As the elements $\std_v - \frac{x_{v}}{x_{\pred(v)} }\std_{\pred(v)}\in H$
    for $v\in V\setminus\{r\}$
    are linearly independent, $\dim(\Adj(\Gamma)_x) \le 1$.
    Thus, $\dim(\Adj(\Gamma)_x) = 1$.
  \item
    \label{lem:dim_Adjx_b}
    If $U = V$ but $\Gamma$ is possibly disconnected, then \ref{lem:dim_Adjx_a} and
    Lemma~\ref{lem:Adj_disjoint_union} show that $\dim(\Adj(\Gamma)_x) =
    \compcnt_\Gamma(U)$ is the number of connected components of $\Gamma$,
    as claimed.
  \item
    For the general case, let $x[U]:= \sum_{u\in U}x_u\std_u \in K U$ be the
    image of $x$ under the natural projection $K V = K U \oplus  K(V\setminus U) \to K U$.
    We claim that
    $\Adj(\Gamma)_x \approx
      \Adj(\Gamma[U])_{x[U]} \oplus K(V\setminus \Nbh_\Gamma[U])$.
    Indeed, this follows since $H$ is spanned by the following two types of
    elements:
    \begin{itemize}[$\triangleright$]
    \item $x_u \std_v - x_v \std_u$ for adjacent vertices $u,v\in U$.
    \item $\std_w$ for $w\in \Nbh_\Gamma[U]\setminus U$.
    \end{itemize}
    The claim follows since by \ref{lem:dim_Adjx_b},
    $\dim(\Adj(\Gamma[U])_{x[U]}) = \compcnt_\Gamma(U)$. \qedhere
  \end{enumerate}
\end{proof}

\begin{rem}
  Let $\Gamma$ have $n$ vertices and $c$ connected components.
  Lemma~\ref{lem:dim_Adjx} generalises a well-known
  basic fact: each oriented incidence matrix of $\Gamma$ has rank
  $\rank(\Gamma) = n - c$; see \cite[Prop.~4.3]{Big93}.
  This easily implies the special case $x = \sum_{v\in V} \std_v$ of
  Lemma~\ref{lem:dim_Adjx}.
\end{rem}

\section{Proof of Theorem~\ref{thm:formula}}
\label{s:proof_formula}

We first rephrase Theorem~\ref{thm:formula}.
For a finite group $G$, define a Dirichlet polynomial $\zeta^{\cc}_G(s) =
\sum_{e=1}^\infty \cc_e(G) e^{-s}$; here, $s$ denotes a complex variable.
For almost simple groups, these functions were studied in \cite{LS05}.
Following \cite{Lin19}, we refer to $\zeta^{\cc}_G(s)$ as the \emph{conjugacy
  class zeta function} of $G$.
We note that a different notion of conjugacy class zeta functions,
occasionally denoted using the same notation $\zeta^{\cc}_G(s)$, can also be
found in the literature; see \cite{dS05,BDOP13,ask,ask2}.
Following \cite{cico}, in \S\ref{ss:cc_zetas}, we will refer to the latter
functions as \itemph{class-counting zeta functions}.

Let $\Gamma = (V,E)$ be a graph with $n$ vertices and $m$ edges.
Theorem~\ref{thm:formula} is equivalent to $\zeta^\cc_{\GG_\Gamma(\FF_q)}(s)
= q^m \Biv_\Gamma(q,q^{-1-s})$.

\begin{lemma}
  \label{lem:cc_zeta_via_Adj}
  $\displaystyle \zeta^\cc_{\GG_\Gamma(\FF_q)}(s) = q^{m-n(s+1)}\sum_{x\in \FF_q V} \card{\Adj(\Gamma)_x}^{s+1}$.
\end{lemma}
\begin{proof}
  Write $V = \{v_1,\dotsc,v_n\}$ and $J = \{ (j,k) : 1 \le j < k\le
  n \text{ with } v_j \sim v_k\}$;
  for a ring $R$, we identify $R V = R^n$.
  We assume that $v_{n'+1},\dotsc,v_n$ are the isolated vertices of $\Gamma$.
  Order the elements of $J$ lexicographically to establish a bijection
  between $\{1,\dotsc,m\}$ and~$J$.

  Write $\fh = \fh_\Gamma(\ZZ)$; see \S\ref{ss:centralisers}.
  Let $\fh'$ and $\fz$ denote the derived subalgebra and centre of~$\fh$,
  respectively.
  By Proposition~\ref{prop:centralisers},
  $\fh'$ and $\fz$ are free $\ZZ$-modules of ranks $m$ and $m + n -n'$,
  respectively.
  Moreover, the images of $\std_1,\dotsc,\std_{n'}$ form a $\ZZ$-basis of $\fh/\fz$.
  Proposition~\ref{prop:centralisers} also shows that for each ring $R$, we
  may identify $\fh' \otimes R$ with the derived subalgebra of $\fh\otimes R$,
  and $\fz\otimes R$ with the centre of $\fh\otimes R$.
  
  Suppose that $q = p^f$ for an \itemph{odd} prime $p$.
  As we noted in \S\ref{ss:centralisers},
  $\GG_\Gamma(\FF_q)$ is isomorphic to the group $\exp(\fh_\Gamma(\FF_q))$
  attached to the Lie $\FF_q$-algebra $\fh_\Gamma(\FF_q) = \fh\otimes \FF_q$ via the Lazard 
  correspondence.
  Let $A(X_1,\dotsc,X_{n'}) \in \Mat_{n'\times m}(\ZZ[X_1,\dotsc,X_{n'}])$ be
  the matrix of linear forms whose $(j,k)$th column has precisely two non-zero entries,
  namely $X_k$ and $-X_j$ in rows $j$ and $k$, respectively.
  Let $\ZZ_p$ denote the ring of $p$-adic integers.
  It is readily verified that the image of the matrix 
  $A(X_1,\dotsc,X_{n'})$ over $\ZZ_p[X_1,\dotsc,X_{n'}]$
  is a ``commutator matrix'' (as defined in \cite[Def.~2.1]{O'BV15})
  associated with the finite Lie
  $\ZZ_p$-algebra $\fh \otimes \FF_p$.

  By \cite[Thm~B]{O'BV15} 
  \[
    \cc_{q^i}(\GG_\Gamma(\FF_q)) = \#\{x\in \FF_q^{n'} : \rank_{\FF_q}(A(x)) =
    i)\}
    \cdot q^{n - n' + m - i}.
  \]
  Let $\check A(X)$ be the $m\times n$ matrix over $\ZZ[X] =
  \ZZ[X_1,\dotsc,X_n]$ which is obtained from $A(X_1,\dotsc,X_{n'})^\top$ by
  adding zero columns in positions $n'+1,\dotsc,n$.
  Hence,
  \[
    \cc_{q^i}(\GG_\Gamma(\FF_q)) = \#\{x\in \FF_q^{n} : \rank_{\FF_q}(\check A(x)) =
    i)\}
    \cdot q^{m - i}.
  \]

  By construction,
  $\Adj(\Gamma)_x \approx \Coker(\check A(x))$ for all $x\in \FF_q V =
  \FF_q^n$.
  In particular, for $x\in \FF_q^n$, we have
  $\rank_{\FF_q}(\check A(x)) = i$ if and only if
  $\dim_{\FF_q}(\Adj(\Gamma)_x) = {n-i}$.
  Hence, writing $\alpha_x = \card{\Adj(\Gamma)_x}$,
  we have $\rank_{\FF_q}(\check A(x)) = i$ if and only if
  $q^n\alpha_x^{-1} = q^i$.
  Thus,
  \begin{align*}
    \zeta^{\cc}_{\GG_\Gamma(\FF_q)}(s)
    & = \sum_{i=0}^\infty \cc_{q^i}(\GG_\Gamma(\FF_q)) q^{-is}
    = \sum_{x\in \FF_q V} q^{m-n}\alpha_x \cdot (q^n\alpha_x^{-1})^{-s}\\&
    = q^{m-n(s+1)}
    \sum_{x\in\FF_qV} \alpha_x^{s+1}.
  \end{align*}

  Finally, if $q$ is even, while the statement of \cite[Thm~B]{O'BV15} itself
  is no longer directly applicable (due to its reliance on the Lazard
  correspondence), its proof in \cite[\S\S 3.1, 3.3--3.4]{O'BV15}
  \itemph{does} apply in the present setting, completing the present proof.
  Indeed, the key ingredient that we need is to be able to identify
  $\GG_\Gamma(\FF_q)$ and $\fh \otimes \FF_q$ as sets such that
  two elements commute in the group if and only if they commute in
  the Lie algebra.
  These conditions are satisfied by
  Propositions~\ref{prop:two_graphical_group_schemes}--\ref{prop:centralisers}.
\end{proof}

\begin{proof}[Proof of Theorem~\ref{thm:formula}]
  By combining Lemma~\ref{lem:cc_zeta_via_Adj}
  and Lemma~\ref{lem:dim_Adjx}, we obtain
  \begin{align*}
    \zeta^\cc_{\GG_\Gamma(\FF_q)}(s)
    & =
      q^{m-n(s+1)} \sum_{x\in \FF_q V} \left(q^{\compcnt_\Gamma(\supp(x)) + n -\card{\Nbh_\Gamma[\supp(x)]}}\right)^{s+1}
      \\
    & =
      q^m \sum_{x\in \FF_q V} (q^{-1-s})^{\card{\Nbh_\Gamma[\supp(x)]} -
      \compcnt_\Gamma(\supp(x))} \\
    & = q^m \sum_{U\subset V} (q-1)^{\card U}(q^{-1-s})^{\card{\Nbh_\Gamma[U]} -
      \compcnt_\Gamma(U)}
    \\
    & =
      q^m \Biv_\Gamma(q,q^{-1-s}).
    \qedhere
  \end{align*}
\end{proof}

\section{Graph operations: disjoint unions and joins}
\label{s:operations}

Let $\Gamma_1 = (V_1,E_1)$ and $\Gamma_2 = (V_2,E_2)$ be graphs with $V_1\cap
V_2 = \emptyset$.
Let $\Gamma_i$ have $n_i$ vertices and $m_i$ edges.
The disjoint union $\Gamma_1\oplus \Gamma_2$ and join $\Gamma_1\join \Gamma_2$
(see \S\ref{ss:intro/known}) of $\Gamma_1$ and $\Gamma_2$ are both graphs on
the vertex set $V_1\cup V_2$ with $m_1+m_2$ and $m_1+ m_2 + n_1n_2$ edges,
respectively.

\begin{prop}
  \label{prop:disjoint_union}
  $\Biv_{\Gamma_1\oplus \Gamma_2}(X,Y) = \Biv_{\Gamma_1}(X,Y)
  \Biv_{\Gamma_2}(X,Y)$.
\end{prop}
\begin{proof}
  This follows since if $U_i\subset V_i$ for $i=1,2$,
  then $\Nbh_{\Gamma_1\oplus\Gamma_2}[U_1\cup U_2] = \Nbh_{\Gamma_1}[U_1] \cup
  \Nbh_{\Gamma_2}[U_2]$
  and $\compcnt_{\Gamma_1\oplus \Gamma_2}(U_1\cup U_2) =
  \compcnt_{\Gamma_1}(U_1)
  + \compcnt_{\Gamma_2}(U_2)$.
\end{proof}

Proposition~\ref{prop:disjoint_union} also follows, {a fortiori}, from
Theorem~\ref{thm:formula} and the identity $\cc_{e}(G_1\times G_2) =
\sum\limits_{\divides d e} \cc_{d}(G_1)\cc_{e/d}(G_2)$ for finite groups $G_1$
and $G_2$.

\begin{prop}
  \label{prop:joins}
  {\small
    \[
      \Biv_{\Gamma_1\join \Gamma_2}(X,Y)
      =
      1 +
      \Bigl(\Biv_{\Gamma_1}(X,Y)-1\Bigr) Y^{n_2} +
      Y^{n_1} \Bigl(\Biv_{\Gamma_2}(X,Y)-1\Bigr)
      + (X^{n_1}-1)(X^{n_2}-1)Y^{n_1+n_2-1}.
    \]
  }
\end{prop}
\begin{proof}
  Write $\Gamma = \Gamma_1\join \Gamma_2$ and $V = V_1\cup V_2$.
  Let $U_i \subset V_i$ for $i=1,2$ and $U = U_1\cup U_2$.
  We seek to relate the summand $t(U) := (X-1)^{\card U}
  Y^{\card{\Nbh_\Gamma[U]}-\compcnt_\Gamma(U)}$
  in the definition of $\Biv_\Gamma(X,Y)$
  to the summands $t_i(U_i) := (X-1)^{\card {U_i}}
  Y^{\card{\Nbh_{\Gamma_i}[U_i]}-\compcnt_{\Gamma_i}(U_i)}$.
  We consider four cases:
  \begin{enumerate}[1.]
  \item
    If $U_1 = U_2 = U = \emptyset$, then $t(U) = 1$.
  \item
    If $U_1 \not= \emptyset = U_2$,
    then $\Nbh_\Gamma[U] = \Nbh_{\Gamma_1}[U_1] \cup V_2$,
    $\compcnt_\Gamma(U) = \compcnt_{\Gamma_1}(U_1)$,
    and $t(U) = t_1(U_1) Y^{n_2}$.
  \item
    Analogously, if $U_1 = \emptyset \not= U_2$,
    then $t(U) = Y^{n_1} t_2(U_2)$.
  \item
    If $U_1\not= \emptyset \not= U_2$,
    then $\Nbh_{\Gamma}[U] = V$, $\compcnt_\Gamma(U) = 1$, and
    $t(U) = (X-1)^{\card{U_1} + \card{U_2}} Y^{n_1+n_2-1}$.
  \end{enumerate}
  We conclude that
  \begin{align*}
    \Biv_\Gamma(X,Y)
    & =
      1 +
      \Bigl(\Biv_{\Gamma_1}(X,Y)-1\Bigr) Y^{n_2} +
      Y^{n_1} \Bigl(\Biv_{\Gamma_2}(X,Y)-1\Bigr) \\
    & \qquad\qquad +
      {\left(\sum_{\substack{\emptyset \not= U_1\subset V_1\\\emptyset \not=
    U_2\subset V_2}}
      (X-1)^{\card{U_1}} (X-1)^{\card{U_2}}\right)}
    Y^{n_1 + n_2-1}.
    \qedhere
  \end{align*}
\end{proof}

As we will explain in \S\ref{ss:zeta_join}, Proposition~\ref{prop:joins}
is closely related to \cite[Prop.~8.4]{cico}.

\begin{ex}[Complete bipartite graphs]
  \label{ex:bipartite}
  Let $\CG_{a,b} = \DG_a \join \DG_b$ be a complete bipartite graph.
  Recall from Example~\ref{ex:CG_DG} that $\Biv_{\DG_n}(X,Y) = X^n$.
  Therefore, by Proposition~\ref{prop:joins}, $\Biv_{\CG_{a,b}}(X,Y) =
  1 + (X^a-1)Y^b + Y^a(X^b-1) + (X^a-1)(X^b-1)Y^{a+b-1}$.
  Hence,
  \begin{align*}
    \CSP_{\CG_{a,b}}(X,Y) & = X^{(a-1)(b-1)}(X^a-1)(X^b-1)Y^{a+b-1} \\
                          & \phantom= +X^{(a-1)b}(X^a-1)Y^b + X^{a(b-1)}(X^b-1)Y^a +X^{ab}
  \end{align*}
  and
  $f_{\CG_{a,b}}(X) = X^{(a-1)(b-1)}((X^a-1)(X^b-1) + X^{a-1}(X^a-1) +
      X^{b-1}(X^b-1)+X^{a+b-1})$.
  We note that the graphical group $\GG_{\CG_{a,b}}(\ZZ/N\ZZ)$
  is the maximal quotient of class at most~$2$ of the free product
  $(\ZZ/N\ZZ)^a * (\ZZ/N\ZZ)^b$; see \cite[\S 3.4]{cico}.
\end{ex}

\begin{ex}[Stars]
  \label{ex:star}
  As a special case of Example~\ref{ex:bipartite},
  let $\Star_n = \DG_{n} \join\, \Point = \CG_{n,1}$ be a star graph on $n+1$ vertices.
  Then
  $\Biv_{\Star_n}(X,Y) =
  (X^{n+1}-X^n)Y^n + (X^n-1)Y + 1$.
  Hence, $\CSP_{\Star_n}(X,Y) =
  X^{n-1} \cdot \bigl((X^2-X) Y^n + (X^n-1)Y + X\bigr)$
  and $f_{\Star_n}(X) =X^{n-1}(X^n+X^2-1)$.
  
\end{ex}

We record the following consequence of Proposition~\ref{prop:joins} for later use.
\begin{cor}
  \label{cor:univ_join}
  Let $\Gamma_1$ and $\Gamma_2$ be graphs.
  Let $\Gamma_i$ have  $m_i$ edges and $n_i$ vertices.
  Then
  \begin{align}
    f_{\Gamma_1\join\Gamma_2}(X)
        =
      X^{m_1 + m_2 + n_1 n_2}
          & + X^{m_2 + (n_1-1)n_2} (f_{\Gamma_1}(X)-X^{m_1}) \nonumber \\
          & + X^{m_1 + n_1(n_2-1)} (f_{\Gamma_2}(X)-X^{m_2})\nonumber \\
        & \phantom{X^{m_1 + m_2}}
          + X^{m_1 + m_2 +
          (n_1-1)(n_2-1)}(X^{n_1}-1)(X^{n_2}-1).
          \label{eq:univ_join}
    \\ && \tag*{\qedsymbol}
  \end{align}
\end{cor}

Beyond disjoint unions and joins, it would be natural to study the effects of
other graph operations on the polynomials $\Biv_\Gamma(X,Y)$.

\section[The constant and leading term]{The constant and leading term of $\Biv_\Gamma(X,Y)$}
\label{s:const_lt}

Let $\Gamma = (V,E)$ be a graph with $n$ vertices and $m$ edges.
In this section, we primarily view $\Biv_\Gamma(X,Y)$ as a polynomial in $Y$
over $\ZZ[X]$.
Its constant term is easily determined.

\begin{prop}
  \label{prop:isolated}
  $\Biv_\Gamma(X,0) = X^i$, where $i$ is the number of isolated vertices of $\Gamma$.
\end{prop}
\begin{proof}
  As $\Biv_\Point(X,Y) = X$,
  by Proposition~\ref{prop:disjoint_union},
  $\Biv_{\Gamma\oplus\Point}(X,Y) = X\cdot\Biv_\Gamma(X,Y)$.
  We may thus assume that $i = 0$.
  Let $U\subset V$.
  As $\compcnt_\Gamma(U) \le \card U \le \card{\Nbh_\Gamma[U]}$,
  we see that
  $\card{\Nbh_\Gamma[U]} = \compcnt_\Gamma(U)$ if and only if $U$
  consists of isolated vertices.
  This only happens for $U = \emptyset$ whence $\Biv_\Gamma(X,0) = 1$.
\end{proof}

For a group-theoretic interpretation of Proposition~\ref{prop:isolated},
note that $q^i$ is the order of the quotient
$\operatorname{Z}(\GG_\Gamma(\FF_q))/[\GG_\Gamma(\FF_q),\GG_\Gamma(\FF_q)]$.

Recall that $\rank(\Gamma) = n -c$, where $c$ is the number of connected
components of $\Gamma$.

\begin{prop}
  \label{prop:deg_Y}
  $\deg_Y(\Biv_\Gamma(X,Y)) = \rank(\Gamma)$.
\end{prop}
\begin{proof}
  If $\Gamma'$ is any subgraph of $\Gamma$, then $\rank(\Gamma') \le
  \rank(\Gamma)$.
  Let $U\subset V$ and write $\bar U = \Nbh_\Gamma[U]$.
  Since every vertex in $\bar U \setminus U$ is adjacent to some vertex in $U$,
  we have $\compcnt_\Gamma(\bar U) \le \compcnt_\Gamma(U)$.
  Hence, $\card{\bar U} - \compcnt_\Gamma(U) \le \card{\bar{U}} -
  \compcnt_\Gamma(\bar U) = \rank(\Gamma[\bar U]) \le \rank(\Gamma)$.
  Thus, $\deg_Y(\Biv_\Gamma(X,Y)) \le \rank(\Gamma)$.
  The summand corresponding to $U = V$ in \eqref{eq:Biv}
  contributes a term $X^n Y^{\rank(\Gamma)}$ to $\Biv_\Gamma(X,Y)$,
  and this term cannot be cancelled by a summand arising from any proper
  subset.
\end{proof}

For $h(X,Y) = \sum_{ij} a_{ij} X^iY^j\in \ZZ[X,Y]$ with $a_{ij}\in \ZZ$,
write $h(X,Y)\Bigl[Y^j\Bigr] = \sum_{i} a_{ij}X^i$
for the coefficient of $Y^j$ in $h(X,Y)$, regarded as a polynomial in $Y$.
We now consider the leading coefficient
$\Biv_\Gamma(X,Y)\Bigl[Y^{\rank(\Gamma)}\Bigr]$ of $\Biv_\Gamma(X,Y)$ as a
polynomial in~$Y$.
Recall that a \emph{dominating set} of $\Gamma$ is a set $D\subset V$ with
$\Nbh_\Gamma[D] = V$.
If, in addition, $\Gamma[D]$ is connected, then $D$ is a \emph{connected
dominating set}.
Let $\ConnDoms(\Gamma)$ be the set of connected dominating sets of $\Gamma$.
Clearly, $\ConnDoms(\Gamma) \not= \emptyset$ if and only if $\Gamma$ is
connected.

\begin{prop}
  \label{prop:dom}
  Suppose that $n\ge 2$.
  Then
  \begin{equation}
    \label{eq:dom}
    \Biv_\Gamma(X+1,Y)\Bigl[Y^{n-1} \Bigr] = \sum\limits_{D \in
      \ConnDoms(\Gamma)} X^{\card D}.
  \end{equation}
\end{prop}
\begin{proof}
  Let $U\subset V$.
  As $n \ge 2$,
  $\card{\Nbh_\Gamma[U]} - \compcnt_\Gamma(U) = n-1$
  if and only if $\Nbh_\Gamma[U] = V$ and $\compcnt_\Gamma(U) = 1$.
  The latter two conditions are satisfied if and only if $U\in \ConnDoms(\Gamma)$.
\end{proof}

\begin{rem}
  In \cite{ME18}, the right-hand side of \eqref{eq:dom} is referred to as the
  \emph{connected domination polynomial} of $\Gamma$.
  These polynomials are relatives of the widely studied domination
  polynomials of graphs introduced in \cite{AL00} (where they were called
  dominating polynomials).
\end{rem}

\begin{cor}
  \label{cor:lc}
  Suppose that $\Gamma$ does not contain isolated vertices.
  Let $V_1,\dotsc,V_c\subset V$ be the distinct connected components of
  $\Gamma$.
  Then
  \[
    \Biv_\Gamma(X+1,Y)\Bigl[Y^{\rank(\Gamma)} \Bigr] =
    \prod_{i=1}^c \sum\limits_{D_i \in   \ConnDoms(\Gamma[V_i])}
    X^{\card{D_i}}.
    \pushQED{\qed}\qedhere
    \popQED
  \]
\end{cor}

The following is well known.
\begin{thm}[{\cite[{\S A1.1, [GT2]}]{GJ79}}]
  \label{thm:conn_dom_NP}
  The problem of deciding, for a given graph $\Gamma$ and $k\ge 1$, whether
  $\Gamma$ admits a connected dominating set of cardinality at most $k$ is
  NP-complete. 
\end{thm}

\begin{proof}[Proof of Proposition~\ref{prop:hard}]
  Combine Theorem~\ref{thm:conn_dom_NP} and Proposition~\ref{prop:dom}.
\end{proof}

We finish this section by showing that typically $\CSP_\Gamma(0,Y) = 0$.
We first record the following consequence of Proposition~\ref{prop:dom}.
\begin{cor}
  \label{cor:ltree}
  Let $\Gamma$ be a tree with $n\ge 3$ vertices and $\ell$ leaves.
  Then $\Biv_\Gamma(X,Y)\Bigl[Y^{n-1}\Bigr] = (X-1)^{n-\ell} X^\ell$.
\end{cor}
\begin{proof}
  Let
  $V^{\mathsmaller{\text{\faLeaf}}}\subset V$ be the set of leaves of $\Gamma$.
  Using $n\ge 3$, it is easy to see that $\ConnDoms(\Gamma) = \{ U :
  V \setminus V^{\mathsmaller{\text{\faLeaf}}} \subset U \subset V\}$.
  Hence, by by Proposition~\ref{prop:dom},
  $\Biv_\Gamma(X,Y)\Bigl[Y^{n-1}\Bigr] = (X-1)^{n-\ell}\sum_{U'\subset
    V^{\mathsmaller{\text{\faLeaf}}}}(X-1)^{\card{U'}} = (X-1)^{n-\ell}X^\ell$.
\end{proof}

\begin{cor}
  \label{cor:CSP(0,Y)}
  Let $\Gamma$ be an arbitrary graph.
  Then $\CSP_\Gamma(0,Y) = 0$
  unless $\Gamma \approx \CG_2^{\oplus r}$, in which
  case $\CSP_\Gamma(0,Y) = (-1)^rY^r$.
\end{cor}
\begin{proof}
  Using Proposition~\ref{prop:disjoint_union} (and its evident analogue for
  $\CSP_\Gamma(X,Y)$), we may assume that $\Gamma$ is connected.
  If $m > \rank(\Gamma)$, then $X$ divides $\CSP_\Gamma(X,Y)$ by
  Proposition~\ref{prop:deg_Y}.
  Thus, suppose that $m = \rank(\Gamma) = n-1$, i.e.\ $\Gamma$ is a tree.
  Since $\CSP_{\CG_1}(X,Y) = X$ and $\CSP_{\CG_2}(X,Y) = (X^2-1)Y + X$, we may
  assume that $n \ge 3$.
  Corollary~\ref{cor:ltree} then implies that $\CSP_\Gamma(0,Y) = 0$.
\end{proof}

\begin{rem}
  The constant term and leading coefficient of $\Biv_\Gamma(X,Y)$ as a
  polynomial in $X-1$ are easily determined:
  $\Biv_\Gamma(1,Y) = 1$ and $\Biv_\Gamma(X+1,Y) = X^n Y^{\rank(\Gamma)}
  + \mathcal O(X^{n-1})$.
  The constant term of $\Biv_\Gamma(X,Y)$ in $X$, 
  i.e.~the polynomial $\Biv_\Gamma(0,Y) = \sum_{U\subset V}(-1)^{\card U}
  Y^{\card{\Nbh_\Gamma[U]}-\compcnt_\Gamma(U)}$, seems to be more mysterious.
\end{rem}

\section{The degrees of class-counting polynomials}
\label{s:ccp}

In this section, we consider the degrees of class-counting
polynomials~$f_\Gamma(X) = \CSP_\Gamma(X,1)$ (see Theorem~\ref{thm:cico/poly}
and Corollary~\ref{cor:univ_formula}).
As before, let $\Gamma = (V,E)$ be a graph with $m$ edges and $n$ vertices.

\subsection[The invariant eta]{Interpreting $\deg(f_\Gamma(X))$: the invariant $\eta(\Gamma)$}

For $U\subset V$, let $\nbhcnt_\Gamma(U) = \card{\Nbh_\Gamma[U]\setminus U}$,
the number of vertices in $V\setminus U$ with a neighbour in $U$.
Recall that $\compcnt_\Gamma(U)$ denotes the number of connected components of
$\Gamma[U]$.
Define
\begin{equation}
  \label{eq:eta}
  \eta(\Gamma) = \max_{U\subset V}\Bigl(\compcnt_\Gamma(U) - \nbhcnt_\Gamma(U)\Bigr) \ge 0.
\end{equation}

Corollary~\ref{cor:univ_formula} implies
\begin{equation}
  \deg(f_\Gamma(X)) = m + \eta(\Gamma).  \label{eq:unideg}
\end{equation}

Our proof of Proposition~\ref{prop:hard} does not imply that computing
$f_\Gamma(X) = X^m \Biv_\Gamma(X,X^{-1})$ is NP-hard, motivating
Question~\ref{qu:computing_f}.

\begin{question}
  Is there a polynomial-time algorithm for computing $\eta(\Gamma)$?
\end{question}

\begin{rem}
  The author is unaware of previous investigations of the numbers
  $\eta(\Gamma)$ in the literature. 
  At a formal level, $\eta(\Gamma)$ is reminiscent of other
  graph-theoretic invariants such as~\textit{critical independence numbers}~\cite{Zha90}
  (which can be computed in polynomial time).
\end{rem}

In the following, we establish bounds for $\eta(\Gamma)$.
Let $\alpha(\Gamma)$ denote the \emph{independence number} of $\Gamma$, i.e.\
the maximal cardinality of an independent set of vertices.
Clearly,

\begin{equation}
  \label{eq:eta_alpha}
  \eta(\Gamma) \le \max_{U\subset V} \compcnt_\Gamma(U) = \alpha(\Gamma).
\end{equation}

While $\eta(\Gamma)$ can be much smaller than $\alpha(\Gamma)$
(cf.\ Proposition~\ref{prop:matteo}\ref{prop:matteo2}), the bound
$\eta(\Gamma) \le \alpha(\Gamma)$ will be useful in our proof of
Proposition~\ref{prop:degf_bounds} below.

Let $c$ be the number of connected components of $\Gamma$.
The case $U = V$ in \eqref{eq:eta} shows that $\eta(\Gamma) \ge c$.
Since $\eta(\Gamma_1\oplus\Gamma_2) = \eta(\Gamma_1) + \eta(\Gamma_2)$, we may
assume that $\Gamma$ is connected.

\begin{prop}
  \label{prop:eta_star}
  Let $\Gamma$ be connected and $n\ge 4$.
  Then $\eta(\Gamma) \le n-2$ with equality if and only if $\Gamma\approx \Star_{n-1}$. 
\end{prop}
\begin{proof}
  For $U\in\{\emptyset, V\}$, we have $\compcnt_\Gamma(U) - \nbhcnt_\Gamma(U)
  \le 1 < n-2$.
  Let $U\subset V$ with $\emptyset\not= U\not= V$.
  Then $\compcnt_\Gamma(U) \le n - 1$ and
  $\nbhcnt_\Gamma(U) > 0$ since $\Gamma$ is connected.
  Hence, $\compcnt_\Gamma(U) - \nbhcnt_\Gamma(U) \le n-2$ and $\eta(\Gamma) \le n-2$.
  Moreover, if $\compcnt_\Gamma(U) - \nbhcnt_\Gamma(U) = n-2$, then
  $\compcnt_\Gamma(U) = n-1$ and $\nbhcnt_\Gamma(U) = 1$.
  This is equivalent to $\Gamma$ being a star graph whose centre is the unique
  vertex in $V\setminus U$.
\end{proof}

By Proposition~\ref{prop:eta_star}, $\eta(\Gamma)$ rarely attains its maximal
value among graphs with~$n$ vertices.
In contrast, $\eta(\Gamma) = 1$ occurs frequently.
Note that $\eta(\CG_n) = 1$ by Example~\ref{ex:univ_CG}.
For complete bipartite graphs, we obtain the following.

\begin{prop}
  $\eta(\CG_{a,b}) = \max(1,\abs{a-b})$.
\end{prop}
\begin{proof}
  This follows by inspection from the formula for $f_{\CG_{a,b}}(X)$ in
  Example~\ref{ex:bipartite}.
\end{proof}

Hence, $\eta(\CG_{a,b}) = 1$ if and only if $\abs{a-b}\le 1$.
To obtain further examples of graphs~$\Gamma$ with $\eta(\Gamma) = 1$, recall
that a graph is \emph{claw-free} if it does not contain $\CG_{1,3} \approx
\Star_3$ as an induced subgraph.
The following proposition and its proof are due to Matteo Cavaleri.
The author thanks him for kindly permitting this material to be included here.

\begin{prop}
  \label{prop:matteo}
  \quad
  \begin{enumerate}[(i)]
  \item
    \label{prop:matteo1}
    $\eta(\Gamma) = \max\Bigl(\compcnt_\Gamma(U) + \card U - \card V: U\subset
    V\text{ is a dominating set of } \Gamma\Bigr)$.
  \item
    \label{prop:matteo2}
    If $\Gamma$ is claw-free and connected, then $\eta(\Gamma) = 1$
    and thus $\deg(f_\Gamma(X)) = m + 1$.
  \end{enumerate}
\end{prop}
\begin{proof}
  \quad
  \begin{enumerate}
  \item
    Let $U\subset V$ with $\Nbh_\Gamma[U]\not= V$.
    Let $C\subset V$ be a connected component of
    $\Gamma[V\setminus \Nbh_\Gamma[U]]$.
    Clearly, $\compcnt_\Gamma(U\cup C) = \compcnt_\Gamma(U) + 1$.
    Let $x\in\Nbh_\Gamma[U\cup C]\setminus(U\cup C)$.
    Then $x\sim y$ for some $y\in U \cup C$.
    Suppose that $x\not\in\Nbh_\Gamma[U]$ so that $y\in C$.
    Then $x\in C$ by the definition of $C$.
    This contradiction shows that
    $\Nbh_\Gamma[U\cup C]\setminus(U\cup C)
    \subset\Nbh_\Gamma[U]\setminus U$.
    Hence, $\compcnt_\Gamma(U\cup C) - \nbhcnt_\Gamma(U\cup C) >
    \compcnt_\Gamma(U) - \nbhcnt_\Gamma(U)$.
    It follows that the maximal value of $\compcnt_\Gamma(U)-\nbhcnt_\Gamma(U)$ is
    attained for a dominating set $U$;
    in that case, $\nbhcnt_\Gamma(U) = \card V - \card U$.
  \item
    Let $U\subset V$ be a $\subset$-maximal dominating set with
    $\compcnt_\Gamma(U) + \card U - \card V = \eta(\Gamma)$.
    Suppose that $U\not= V$.
    Choose $x\in V\setminus U$.
    By maximality of $U$, $\compcnt_\Gamma(U) + \card U >
    \compcnt_\Gamma(U\cup \{x\}) + \card U + 1$.
    Hence, there are distinct connected components
    $C_1,C_2,C_3$ of $\Gamma[U]$ such that $\Gamma[C_1\cup C_2\cup
    C_3\cup\{x\}]$ is connected.
    Choose $c_i\in C_i$ with $x\sim c_i$.
    Then $\Gamma[\{c_1,c_2,c_3,x\}] \approx \Star_3$.
    We conclude that if $\Gamma$ is claw-free and connected, then $U=V$
    and thus $\eta(\Gamma) = 1$.
    \qedhere
  \end{enumerate}
\end{proof}

Let $\Delta(\Gamma)$ denote the maximum vertex degree of $\Gamma$.

\begin{lemma}
  \label{lem:tree_eta_Delta}
  Let $\Tau$ be a tree.
  Then $\eta(\Tau) \ge \Delta(\Tau)-1$.
\end{lemma}
\begin{proof}
  Let $w_1,\dotsc,w_d$ be the distinct vertices adjacent to a vertex $u$ of
  $\Tau$.
  Let $W_i$ consist of $w_i$ and all its descendants in the rooted tree
  $(\Tau,u)$.
  Define $W := W_1\cup\dotsb \cup W_d$.
  By construction, $\compcnt_\Tau(W) = d$ and $\nbhcnt_\Tau(W) = 1$ whence
  $\eta(\Tau) \ge d - 1$.
\end{proof}

\begin{cor}
  \label{cor:tree_eta1}
  Let $\Tau$ be a tree.
  Then $\eta(\Tau) = 1$ if and only if $\Tau$ is a path.
\end{cor}
\begin{proof}
  By Lemma~\ref{lem:tree_eta_Delta}, $\eta(\Tau) > 1$ unless $\Tau$ is a path.
  By Proposition~\ref{prop:matteo}\ref{prop:matteo2} or equation~\eqref{eq:f_path},
  we have $\eta(\Path n) = 1$.
\end{proof}

\subsection[Upper and lower bounds]{Upper and lower bounds for $\deg(f_\Gamma(X))$}

We obtain sharp bounds for $\deg(f_\Gamma(X))$ as $\Gamma$ ranges over all
graphs with $n$ vertices.

\begin{prop}
  \label{prop:degf_bounds}
  Let $\Gamma$ be a graph with $n \ge 1$ vertices.
  Then
  $n \le \deg(f_\Gamma(X)) \le \binom n 2 + 1$.
  The lower bound is attained if and only if $\Gamma$ is a disjoint union of
  paths.
  The upper bound is attained if and only if $\Gamma$ is complete or $n = 2$.
\end{prop}

Our proof of Proposition~\ref{prop:degf_bounds} will rely on an upper bound
for independence numbers.

\begin{lemma}[{\cite{Han79}}]
  \label{lem:hansen}
  Let $\Gamma$ be a graph with $m$ edges and $n$ vertices.
  Then
  \[
    \alpha(\Gamma) \le \left\lfloor \frac 1 2 + \sqrt{\frac 1 4 + n^2 - n-2m}
    \right\rfloor.
  \]
\end{lemma}

\begin{proof}[Proof of Proposition~\ref{prop:degf_bounds}]
  As before, let $m$ denote the number of edges of $\Gamma = (V,E)$.
  \begin{enumerate}
  \item
    \textit{Lower bound.}
    We may assume that $\Gamma$ is connected so that $m\ge n-1$.
    As $\eta(\Gamma) \ge 1$, equation~\eqref{eq:unideg}
    shows that $\deg(f_\Gamma(X)) \ge n$ with
    equality if and only if $\Gamma$ is a tree and $\eta(\Gamma) = 1$.
    By Corollary~\ref{cor:tree_eta1}, the latter condition is equivalent to
    $\Gamma\approx \Path n$.
  \item
    \textit{Upper bound.}
    Since $\deg(f_{\CG_n}(X)) = \binom n 2 + 1$ by Example~\ref{ex:univ_CG},
    it suffices to show that $\deg(f_\Gamma(X)) \le \binom n 2$ whenever $m <
    \binom n 2$.
    We may assume that $n\ge 3$.
    Writing $m = \binom n 2 - k$, Lemma~\ref{lem:hansen} shows that
    $\alpha(\Gamma)
    \le \left\lfloor \frac 1 2 + \sqrt{2k+\frac 1 4}
    \right\rfloor$. 
    Hence, if $k\ge 2$,
    then $\alpha(\Gamma) \le k$.
    By equation \eqref{eq:eta_alpha},
    $\deg(f_\Gamma(X)) = m + \eta(\Gamma) \le m + \alpha(\Gamma) \le m + k = \binom n
    2$.
    For $k = 1$, we have $\Gamma\approx\DG_2 \join \CG_{n-2}$ and
    $\deg(f_{\Gamma}(X)) = \binom n 2$ by Proposition~\ref{prop:matteo}\ref{prop:matteo2}.
    (Alternatively, we may combine Corollary~\ref{cor:univ_join} and
    Example~\ref{ex:univ_CG}.) 
    \qedhere
  \end{enumerate}
\end{proof}

\section{Applications to zeta functions of graphical group schemes}
\label{s:app}

We briefly relate some of our findings to recent work on zeta functions of
groups.

\subsection{Reminder: class-counting and conjugacy class zeta functions}
\label{ss:cc_zetas}

The study of zeta functions associated with groups and group-theoretic
counting problems goes back to influential work of Grunewald et
al.~\cite{GSS88}.
Let $\GG$ be a group scheme of finite type over a compact discrete valuation
ring $\fO$ with maximal ideal $\fP$.
The \emph{class-counting zeta function} of $\GG$ is the Dirichlet series
$\zeta^\concnt_\GG(s) = \sum_{i=0}^\infty
\concnt(\GG(\fO/\fP^i))\card{\fO/\fP^i}^{-s}$.
Beginning with work of du~Sautoy~\cite{dS05}, these and closely related series
enumerating conjugacy classes have recently been studied, see \cite{BDOP13,
ask, ask2, Lin19, Lin20, cico}.
Recall the definition of the conjugacy class zeta function $\zeta^\cc_G(s)$
associated with a finite group $G$ from~\S\ref{s:proof_formula}.
Lins~\cite[Def.~1.2]{Lin19} introduced a refinement of $\zeta^\concnt_\GG(s)$,
the \emph{bivariate conjugacy class zeta function} $\zeta^\cc_{\GG}(s_1,s_2)
=\sum_{i=0}^\infty \zeta_{\GG(\fO/\fP^i)}^\cc(s_1) \card{\fO/\fP^i}^{-s_2}$ of
$\GG$ and studied these functions for certain classes of unipotent group
schemes; note that $\zeta^\concnt_\GG(s) = \zeta^\cc_\GG(0,s)$.

Theorem~\ref{thm:cico/poly} is in fact a special case of a far more general
result pertaining to class-counting zeta functions associated with graphical
group schemes.

\begin{thm}[{Cf.~\cite[Cor.~B]{cico}}]
  \label{thm:cico/uniform}
  For each graph $\Gamma$, there exists a rational function $\tilde W_\Gamma(X,Y) \in
  \QQ(X,Y)$
  with the following property:
  for each compact discrete valuation ring $\fO$ with residue field size $q$,
  we have
  $\zeta^\concnt_{\GG_\Gamma\otimes \fO}(s)= \tilde W_\Gamma(q,q^{-s})$.
\end{thm}

Theorem~\ref{thm:cico/uniform} contains Theorem~\ref{thm:cico/poly} as a
special case via $\tilde W_\Gamma(X,Y) = 1 + f_\Gamma(X) Y + \mathcal O(Y^2)$.

\begin{rem}
  In the present article, we chose to normalise our polynomials and rational
  functions slightly differently compared to \cite{cico}.
  Namely, what we call $\tilde W_\Gamma(X,Y)$ here coincides with
  $W_\Gamma^-(X,X^mY)$ in \cite{cico}, where $m$ is the number of edges of
  $\Gamma$.
\end{rem}

\subsection{Class-counting zeta functions of graphical group schemes and
  joins}
\label{ss:zeta_join}

Let $\Gamma_1$ and $\Gamma_2$ be graphs with $n_1$ and $n_2$ vertices and
$m_1$ and $m_2$ edges, respectively.
Define a rational function $Q_{\Gamma_1,\Gamma_2}(X,Y) \in \QQ(X,Y)$ via

{\footnotesize
\begin{align*}
  Q_{\Gamma_1,\Gamma_2}(X,Y)
  & =
    X^{m_1 + m_2 + (n_1-1)(n_2-1)}Y - 1 \\
  & \phantom= \, + \tilde W_{\Gamma_1}(X,X^{m_2 + (n_1-1)n_2}Y)\cdot
    (1-X^{m_1 + m_2 + (n_1-1)n_2}Y)
    (1-X^{m_1 + m_2 + (n_1-1)n_2+1}Y)\\
  & \phantom= \, + \tilde W_{\Gamma_2}(X,X^{m_1 + n_1(n_2-1)}Y)\cdot
    (1-X^{m_1 + m_2 + n_1(n_2-1)}Y)
    (1-X^{m_1 + m_2 + n_1(n_2-1)+1}Y).
\end{align*}}

Our study of joins in \S\ref{s:operations} was motivated by the following.

\begin{thm}[{\cite[Prop.~8.4]{cico}}]
  \label{thm:cico/join}
  Suppose that $\Gamma_1$ and $\Gamma_2$ are cographs.
  Then
  {\begin{equation}
      \label{eq:cico/join}
      \tilde W_{\Gamma_1\join \Gamma_2}(X,Y)
      =
      \frac{Q_{\Gamma_1,\Gamma_2}(X,Y)}{(1-X^{m_1+m_2+n_1n_2}Y)(1-X^{m_1+m_2+n_1n_2+1}Y)}.
  \end{equation}}
\end{thm}

It remains unclear whether the assumption that $\Gamma_1$ and $\Gamma_2$ be
cographs in Theorem~\ref{thm:cico/join} is truly needed or if it is merely an
artefact of the proof given in \cite{cico}. 

\begin{question}[{\cite[Question 10.1]{cico}}]
  \label{q:cico/join}
  Does \eqref{eq:cico/join} hold for arbitrary graphs $\Gamma_1$ and
  $\Gamma_2$?
\end{question}

We obtain a positive answer to a (much weaker!) ``approximate form'' of
Question~\ref{q:cico/join}.

\begin{prop}
  Let $\Gamma_1$ and $\Gamma_2$ be arbitrary graphs with $n_1$ and $n_2$
  vertices and $m_1$ and $m_2$ edges, respectively.
  Then, regarded as formal power series in $Y$ over $\QQ(X)$,
  the rational function $\tilde W_{\Gamma_1\join \Gamma_2}(X,Y)$ and the right-hand
  side of \eqref{eq:cico/join} agree modulo $Y^2$.
\end{prop}
\begin{proof}
  This follows from Corollary~\ref{cor:univ_join}: by expanding
  the right-hand side of \eqref{eq:cico/join} as a series in $Y$, we find that
  the coefficient of $Y$ is given by the right-hand side of
  \eqref{eq:univ_join}.
\end{proof}

\subsection{Uniformity and the difficulty of computing zeta functions of groups}
\label{ss:uniformity}

For reasons that are not truly understood at present, many interesting
examples of zeta functions associated with group-theoretic counting problems
are ``(almost) uniform''.
As we now recall, the task of symbolically computing such zeta functions is
well-defined.

\paragraph{Uniformity.}
Beginning with a global object $G$ and a type of counting problem, we often
obtain (a) associated local objects $G_p$ indexed by primes (or places) $p$
and (b)~associated local zeta functions $\zeta_{G_p}(s)$. 
The family $(\zeta_{G_p}(s))_p$ of zeta functions is \emph{(almost) uniform}
if there exists $W_G(X,Y)\in \QQ(X,Y)$ such that $\zeta_{G_p}(s) =
W_G(p,p^{-s})$ for (almost) all $p$.
(Stronger forms of uniformity may also take into account local base extensions
or changing the characteristic of compact discrete valuation rings under
consideration.
Variants apply to multivariate zeta functions such as
$\zeta^\cc_{\GG}(s_1,s_2)$.)
It is then natural to seek to devise algorithms for computing $W_G(X,Y)$ and
to consider the complexity of such algorithms.

Numerous computations of (almost) uniform zeta functions associated with
groups and related algebraic structures have been recorded in the literature;
see e.g.\ \cite{dSW08}.
For a recent example, Carnevale et al.~\cite{CSV19} (see also \cite{CSV20})
obtained strong uniformity results for ideal zeta functions of certain
nilpotent Lie rings.
Their explicit formulae for rational functions as sums over chain complexes
involve sums of super-exponentially many rational functions.

The following example illustrates how class-counting zeta functions associated
with graphical group schemes fit the above template for uniformity of zeta
functions.

\begin{ex}
  \label{ex:graphical_uniform}
  Let $G = \GG_\Gamma$ be a graphical group scheme.
  For a prime $p$, let $\ZZ_p$ denote the ring of $p$-adic integers and let
  $G_p = G \otimes \ZZ_p$.
  Writing $\zeta_{G_p}(s) = \zeta^{\concnt}_{G_p}(s)$, the family
  $(\zeta_{G_p}(s))_p$ is uniform by Theorem~\ref{thm:cico/uniform} with
  $W_G(X,Y) = \tilde W_{\Gamma}(X,Y)$.
  The constructive proof of Theorem~\ref{thm:cico/uniform} in \cite{cico}
  gives rise to an algorithm for computing $\tilde W_\Gamma(X,Y)$ (see \cite[\S 9.1]{cico}).
  While no complexity analysis was carried out in \cite{cico}, 
  this algorithm appears likely to be substantially worse than polynomial-time.
  For a cograph $\Gamma$, \cite[Thms~C--D]{cico} combine to 
  produce a formula for $\tilde W_\Gamma(X,Y)$ as a sum of explicit rational
  functions, the number of which grows super-exponentially with the number of
  vertices of $\Gamma$.
\end{ex}

\paragraph{Computing bivariate conjugacy class zeta functions.}
As indicated (but not spelled out as such) in \cite[\S 8.5]{cico},
Theorem~\ref{thm:cico/uniform} admits the following generalisation: given a
graph~$\Gamma$, there exists $\tilde W_\Gamma(X,Y,Z)\in \QQ(X,Y,Z)$ such that
for all compact discrete valuation rings with residue field size $q$,
$\zeta^\cc_{\GG_\Gamma}(s_1,s_2) = \tilde W_\Gamma(q,q^{-s_1},q^{-s_2})$.
(Hence, $\tilde W_\Gamma(X,1,Z) = \tilde W_\Gamma(X,Z)$.)
Suppose that, given $\Gamma$, an oracle provided us with $\tilde
W_\Gamma(X,Y,Z)$ as a reduced fraction of polynomials.
Since $\tilde W_\Gamma(X,Y,Z) = 1 + \CSP_\Gamma(X,Y) Z + \mathcal O(Z^2)$, we
may then compute $\CSP_\Gamma(X,Y)$ by symbolic differentiation.
In particular, Proposition~\ref{prop:hard} implies that computing $\tilde
W_\Gamma(X,Y,Z)$ is NP-hard.
To the author's knowledge, this is the first non-trivial \itemph{lower} bound
for the difficulty of computing uniform zeta functions associated with groups.
We do not presently obtain a similar lower bound for the difficulty
of computing $\tilde W_\Gamma(X,Y)$ since the difficulty of determining
$f_\Gamma(X)$ remained unresolved in~\S\ref{s:ccp}.

\subsection{Open problem: higher congruence levels}

It is an open problem to find a combinatorial formula for the
rational functions $\tilde W_\Gamma(X,Y)$
(or their generalisations $\tilde W_\Gamma(X,Y,Z)$ from \S\ref{ss:uniformity}) as
$\Gamma$ ranges over all graphs on a given vertex set;
cf.\ \cite[Question\ 1.8(iii)]{cico}.
Corollary~\ref{cor:univ_formula} provides such a formula for the first
non-trivial coefficient of $\tilde W_\Gamma(X,Y) = 1 + f_\Gamma(X) Y +
\mathcal O(Y^2)$, and Theorem~\ref{thm:formula} provides a formula for
the first non-trivial coefficient of $\tilde W_\Gamma(X,Y,Z)= 1 +
\CSP_\Gamma(X,Y) Z + \mathcal O(Z^2)$.
As suggested by one of the anonymous referees, it is natural to ask whether
a combinatorial formulae of the type considered here can be obtained for the
coefficient of $Y^2$ in $\tilde W_\Gamma(X,Y)$ or of $Z^2$ in $\tilde
W_\Gamma(X,Y,Z)$.
These coefficients enumerate conjugacy classes of graphical
groups $\GG_\Gamma(\fO/\fP^2)$, where $\fO$ is a compact discrete valuation
ring with maximal ideal $\fP$.
The ``dual'' problem of enumerating characters (see \S\ref{ss:Irr}) is
related to recent research developments.
In particular, the character theory of
reductive groups over rings of the form $\fO/\fP^2$ has received considerable
attention; see e.g.~\cite{Sing10,SV19}.

\subsection*{Acknowledgements}

I am grateful to Matteo Cavaleri, Yinan Li, and Christopher Voll for
discussions on the work described in this paper,
and to the anonymous referees for numerous helpful comments and suggestions.


{
  \def\emph{\itemph}
  \bibliographystyle{abbrv}
  \bibliography{csp}
}

\vspace*{2em}

{\small
  \noindent
  School of Mathematical and Statistical Sciences \\
  National University of Ireland, Galway \\
  Ireland\\
  \\
  E-mail: \href{mailto:tobias.rossmann@nuigalway.ie}{tobias.rossmann@nuigalway.ie}
}

\end{document}